\documentclass[12pt]{amsart}

\usepackage{amsmath,amsthm,amsfonts,amssymb,latexsym}

\usepackage{hyperref}

\usepackage{enumerate}
\usepackage{color}
\usepackage[dvipsnames]{xcolor}

\usepackage[shortlabels]{enumitem}

\usepackage[utf8]{inputenc}

\usepackage {parskip}
\begin{document}

\theoremstyle{plain}

\newtheorem{thm}{Theorem}[section]

\newtheorem{lem}[thm]{Lemma}
\newtheorem{Problem B}[thm]{Problem B}

\newtheorem{thml}{Theorem}
\renewcommand*{\thethml}{\Alph{thml}}   
\newtheorem{corl}{Corollary}
\renewcommand*{\thecorl}{C}

\newtheorem{pro}[thm]{Proposition}
\newtheorem{conj}[thm]{Conjecture}
\newtheorem{cor}[thm]{Corollary}
\newtheorem{que}[thm]{Question}
\newtheorem{rem}[thm]{Remark}
\newtheorem{defi}[thm]{Definition}

\newtheorem*{thmA}{Theorem A}
\newtheorem*{thmB}{Theorem B}
\newtheorem*{corC}{Corollary C}
\newtheorem*{thmC}{Theorem C}
\newtheorem*{thmD}{Theorem D}
\newtheorem*{thmE}{Theorem E}
 
\newtheorem*{thmAcl}{Theorem A$^{*}$}
\newtheorem*{thmBcl}{Theorem B$^{*}$}
\newcommand{\dd}{\mathrm{d}}

\newcommand{\Maxn}{\operatorname{Max_{\textbf{N}}}}
\newcommand{\Syl}{\operatorname{Syl}}
\newcommand{\Lin}{\operatorname{Lin}}
\newcommand{\U}{\mathbf{U}}
\newcommand{\R}{\mathbf{R}}
\newcommand{\dl}{\operatorname{dl}}
\newcommand{\Con}{\operatorname{Con}}
\newcommand{\cl}{\operatorname{cl}}
\newcommand{\Stab}{\operatorname{Stab}}
\newcommand{\Aut}{\operatorname{Aut}}
\newcommand{\Ker}{\operatorname{Ker}}
\newcommand{\InnDiag}{\operatorname{InnDiag}}
\newcommand{\fl}{\operatorname{fl}}
\newcommand{\Irr}{\operatorname{Irr}}
\newcommand{\FF}{\mathbb{F}}
\newcommand{\EE}{\mathbb{E}}
\newcommand{\normal}{\trianglelefteq}
\newcommand{\sn}{\normal\normal}
\newcommand{\Bl}{\mathrm{Bl}}
\newcommand{\NN}{\mathbb{N}}
\newcommand{\N}{\mathbf{N}}
\newcommand{\bfC}{\mathbf{C}}
\newcommand{\bfO}{\mathbf{O}}
\newcommand{\bfF}{\mathbf{F}}
\def\GGG{{\mathcal G}}
\def\HHH{{\mathcal H}}
\def\HH{{\mathcal H}}
\def\irra#1#2{{\rm Irr}_{#1}(#2)}

\renewcommand{\labelenumi}{\upshape (\roman{enumi})}

\newcommand{\PSL}{\operatorname{PSL}}
\newcommand{\PSU}{\operatorname{PSU}}
\newcommand{\alt}{\operatorname{Alt}}

\providecommand{\V}{\mathrm{V}}
\providecommand{\E}{\mathrm{E}}
\providecommand{\ir}{\mathrm{Irm_{rv}}}
\providecommand{\Irrr}{\mathrm{Irm_{rv}}}
\providecommand{\re}{\mathrm{Re}}

\numberwithin{equation}{section}
\def\irrp#1{{\rm Irr}_{p'}(#1)}

\def\ibrrp#1{{\rm IBr}_{\Bbb R, p'}(#1)}
\def\C{{\mathbb C}}
\def\Q{{\mathbb Q}}
\def\irr#1{{\rm Irr}(#1)}
\def\irrp#1{{\rm Irr}_{p^\prime}(#1)}
\def\irrq#1{{\rm Irr}_{q^\prime}(#1)}
\def \c#1{{\cal #1}}
\def \aut#1{{\rm Aut}(#1)}
\def\cent#1#2{{\bf C}_{#1}(#2)}
\def\norm#1#2{{\bf N}_{#1}(#2)}
\def\zent#1{{\bf Z}(#1)}
\def\syl#1#2{{\rm Syl}_#1(#2)}
\def\normal{\triangleleft\,}
\def\oh#1#2{{\bf O}_{#1}(#2)}
\def\Oh#1#2{{\bf O}^{#1}(#2)}
\def\det#1{{\rm det}(#1)}
\def\gal#1{{\rm Gal}(#1)}
\def\fit#1{{\bf F}(#1)}
\def\ker#1{{\rm ker}(#1)}
\def\normalm#1#2{{\bf N}_{#1}(#2)}
\def\alt#1{{\rm Alt}(#1)}
\def\iitem#1{\goodbreak\par\noindent{\bf #1}}
   \def \mod#1{\, {\rm mod} \, #1 \, }
\def\sbs{\subseteq}

\def\gc{{\bf GC}}
\def\ngc{{non-{\bf GC}}}
\def\ngcs{{non-{\bf GC}$^*$}}
\newcommand{\notd}{{\!\not{|}}}

\newcommand{\Z}{\mathbf{Z}}
\newcommand{\Out}{{\mathrm {Out}}}
\newcommand{\Mult}{{\mathrm {Mult}}}
\newcommand{\Inn}{{\mathrm {Inn}}}
\newcommand{\IBR}{{\mathrm {IBr}}}
\newcommand{\IBRL}{{\mathrm {IBr}}_{\ell}}
\newcommand{\IBRP}{{\mathrm {IBr}}_{p}}
\newcommand{\cd}{\mathrm{cd}}
\newcommand{\ord}{{\mathrm {ord}}}
\def\id{\mathop{\mathrm{ id}}\nolimits}
\renewcommand{\Im}{{\mathrm {Im}}}
\newcommand{\Ind}{{\mathrm {Ind}}}
\newcommand{\diag}{{\mathrm {diag}}}
\newcommand{\soc}{{\mathrm {soc}}}
\newcommand{\End}{{\mathrm {End}}}
\newcommand{\sol}{{\mathrm {sol}}}
\newcommand{\Hom}{{\mathrm {Hom}}}
\newcommand{\Mor}{{\mathrm {Mor}}}
\newcommand{\Mat}{{\mathrm {Mat}}}
\def\rank{\mathop{\mathrm{ rank}}\nolimits}
\newcommand{\Tr}{{\mathrm {Tr}}}
\newcommand{\tr}{{\mathrm {tr}}}
\newcommand{\Gal}{{\rm Gal}}
\newcommand{\Spec}{{\mathrm {Spec}}}
\newcommand{\ad}{{\mathrm {ad}}}
\newcommand{\Sym}{{\mathrm {Sym}}}
\newcommand{\Char}{{\mathrm {Char}}}
\newcommand{\pr}{{\mathrm {pr}}}
\newcommand{\rad}{{\mathrm {rad}}}
\newcommand{\abel}{{\mathrm {abel}}}
\newcommand{\PGL}{{\mathrm {PGL}}}
\newcommand{\PCSp}{{\mathrm {PCSp}}}
\newcommand{\PGU}{{\mathrm {PGU}}}
\newcommand{\codim}{{\mathrm {codim}}}
\newcommand{\ind}{{\mathrm {ind}}}
\newcommand{\Res}{{\mathrm {Res}}}
\newcommand{\Lie}{{\mathrm {Lie}}}
\newcommand{\Ext}{{\mathrm {Ext}}}
\newcommand{\Alt}{{\mathrm {Alt}}}
\newcommand{\AAA}{{\sf A}}
\newcommand{\SSS}{{\sf S}}
\newcommand{\DDD}{{\sf D}}
\newcommand{\QQQ}{{\sf Q}}
\newcommand{\CCC}{{\sf C}}
\newcommand{\SL}{{\mathrm {SL}}}
\newcommand{\Sp}{{\mathrm {Sp}}}
\newcommand{\PSp}{{\mathrm {PSp}}}
\newcommand{\SU}{{\mathrm {SU}}}
\newcommand{\GL}{{\mathrm {GL}}}
\newcommand{\GU}{{\mathrm {GU}}}
\newcommand{\Spin}{{\mathrm {Spin}}}
\newcommand{\CC}{{\mathbb C}}
\newcommand{\CB}{{\mathbf C}}
\newcommand{\RR}{{\mathbb R}}
\newcommand{\QQ}{{\mathbb Q}}
\newcommand{\ZZ}{{\mathbb Z}}
\newcommand{\bfN}{{\mathbf N}}
\newcommand{\bfZ}{{\mathbf Z}}
\newcommand{\PP}{{\mathbb P}}
\newcommand{\cG}{{\mathcal G}}
\newcommand{\cH}{{\mathcal H}}
\newcommand{\cQ}{{\mathcal Q}}
\newcommand{\GA}{{\mathfrak G}}
\newcommand{\cT}{{\mathcal T}}
\newcommand{\cL}{{\mathcal L}}
\newcommand{\IBr}{\mathrm{IBr}}
\newcommand{\cS}{{\mathcal S}}
\newcommand{\cR}{{\mathcal R}}
\newcommand{\GCD}{\GC^{*}}
\newcommand{\fA}{\mathfrak{A}}
\newcommand{\TCD}{\TC^{*}}
\newcommand{\FD}{F^{*}}
\newcommand{\GD}{G^{*}}
\newcommand{\HD}{H^{*}}
\newcommand{\GCF}{\GC^{F}}
\newcommand{\TCF}{\TC^{F}}
\newcommand{\PCF}{\PC^{F}}
\newcommand{\GCDF}{(\GC^{*})^{F^{*}}}
\newcommand{\RGTT}{R^{\GC}_{\TC}(\theta)}
\newcommand{\RGTA}{R^{\GC}_{\TC}(1)}
\newcommand{\Om}{\Omega}
\newcommand{\eps}{\epsilon}
\newcommand{\varep}{\varepsilon}
\newcommand{\al}{\alpha}
\newcommand{\chis}{\chi_{s}}
\newcommand{\sigmad}{\sigma^{*}}
\newcommand{\PA}{\boldsymbol{\alpha}}
\newcommand{\gam}{\gamma}
\newcommand{\lam}{\lambda}
\newcommand{\la}{\langle}
\newcommand{\genf}{F^*}
\newcommand{\ra}{\rangle}
\newcommand{\hs}{\hat{s}}
\newcommand{\htt}{\hat{t}}
\newcommand{\tG}{\hat G}
\newcommand{\St}{\mathsf {St}}
\newcommand{\bfs}{\boldsymbol{s}}
\newcommand{\bfl}{\boldsymbol{\lambda}}
\newcommand{\tn}{\hspace{0.5mm}^{t}\hspace*{-0.2mm}}
\newcommand{\ta}{\hspace{0.5mm}^{2}\hspace*{-0.2mm}}
\newcommand{\tb}{\hspace{0.5mm}^{3}\hspace*{-0.2mm}}
\def\skipa{\vspace{-1.5mm} & \vspace{-1.5mm} & \vspace{-1.5mm}\\}
\newcommand{\tw}[1]{{}^#1\!}
\renewcommand{\mod}{\bmod \,}

\marginparsep-0.5cm

\renewcommand{\thefootnote}{\fnsymbol{footnote}}
\footnotesep6.5pt
\title{Characters of prime power degree in principal blocks}

\author[]{J. Miquel Mart\'inez}
\address{Departament de Matem\`atiques, Universitat de Val\`encia, 46100 Burjassot, Val\`encia, Spain}
\email{josep.m.martinez@uv.es}

\thanks{This research is partially supported by Ministerio de Ciencia e Innovaci\'on PID2019-103854GB-I00, Generalitat Valenciana CIAICO/2021-163, as well as
a fellowship UV-INV-PREDOC20-1356056 from Universitat de Val\`encia and a travel grant associated to the same fellowship}



\begin{abstract} 
 We describe finite groups whose principal block contains only characters of prime power degree.
 \end{abstract}

\maketitle
\section{Introduction}

Let $G$ be a finite group and let $\cd(G)$ denote the set of degrees of its irreducible ordinary characters. The properties of $G$ that can be seen in $\cd(G)$ have been a subject of study for over half a century. Some fundamental results in this topic are the Isaacs--Passman theorems, Thompson's theorem and the It\^o--Michler theorem, which was one of the first applications of the classification of finite simple groups to character theory. The recent survey \cite{Mor23} gives a good overview of the open problems and the techniques used to explore them.

 In \cite{Man85a} and \cite{Man85b}, O. Manz described the structure of finite groups all whose characters have prime power degrees. He proved that these groups are either solvable or isomorphic to a direct product $A\times S$ where $A$ is an abelian group and $S$ is isomorphic to $\SL_2(4)\cong \fA_5$ or $\SL_2(8)$ (although this last part depends on Brauer's height zero conjecture, which has been proved recently in \cite{MNST}). A version of Manz's results for Brauer characters was obtained in \cite{Ton14} and for real character degrees in \cite{Bon22} and \cite{CT22}.
 
Let $p$ be a prime and let $B_0(G)$ be the principal $p$-block of $G$. If every nonlinear character in $B_0(G)$ has degree divisible by $p$ then it is well known that $G$ has a normal $p$-complement \cite{IS} (this is in fact a principal block version of Thompson's theorem). If $\cd(B_0(G))$ consists only of powers of a different prime $q\neq p$ then it was proved in \cite[Theorem 5.3]{M21} that $G/\oh{p'}G$ has a normal abelian $q$-complement. In view of these results, A. Moret\'o asked the author if it is possible to describe finite groups $G$ such that $\cd(B_0(G))$ consists of powers of possibly different primes, providing a principal block version to Manz's results.
Of course, in this case $p$-solvability is no longer guaranteed (take $G=\fA_5$ and $p$ any prime dividing $|G|$), but it is possible to accurately describe the structure of $G$ and this description is the purpose of this note.

Since this property is inherited by factor groups and normal subgroups (see Lemma \ref{lem:normalsubgroups}), we inevitably run into the problem of determining which finite simple groups satisfy our hypothesis. The following completely describes these groups.

\begin{thml}\label{thm:simple}
Let $S$ be a nonabelian finite simple group, and let $p$ be a prime dividing $|S|$. Then $\cd(B_0(S))$ contains only prime powers if and only if $(S, p)$ is one of the following
\begin{enumerate}
\item $S=\PSL_2(q)$ for $q$ a Fermat or Mersenne prime and $p\not\in\{2, q\}$,
\item $S=\SL_2(2^{n})$ where $q=2^{n}\pm 1$ is a prime, $p\not\in \{2, q\}$,
\item $S=\SL_2(8)$ and $p\in\{2, 3, 5, 7\}$,
\item $S=\SL_2(4)\cong \fA_5$ and $p\in\{2,3,5\}$,
\item $S=\PSL_2(9)\cong \fA_6$ and $p=5$.
\end{enumerate}
and in all cases there are exactly two primes dividing the degrees in $\cd(B_0(S))$.
\end{thml}

Most of the work towards the proof of Theorem \ref{thm:simple} follows from the results of \cite{MZ01} and \cite{BBOO01}, where the prime power degree characters of finite (quasi-)simple groups were determined. In fact, Theorem \ref{thm:simple} is fairly simple to obtain using these results and the work done in \cite{RSV20} and \cite{GRSS20} on simple groups of Lie type.

For general finite groups we have the following description.

\begin{thml}\label{thm:main}
Let $G$ be a finite group, $p$ a prime dividing $|G|$ and assume $\cd(B_0(G))$ consists only of prime powers. Then one of the following happens:
\begin{enumerate}
\item $G/\oh{p'}G$ is a solvable group described in \cite{Man85a},
\item there is a normal subgroup $\oh{p'}G\sbs M\normal G$ such that $$M/\oh{p'}G=H \times S$$ where $H$ is an abelian $p$-group and $(S, p)$ is one of the pairs from Theorem \ref{thm:simple}. Further, $G/M$ is isomorphic to a subgroup of $\Out(S)$.
\end{enumerate}
\end{thml}

We remark that the only case where $G/M$ is not cyclic is when $S\cong \fA_6$. Indeed, the characters in the principal $5$-block of $\Aut(\fA_6)$ have degrees $1, 9$ and $16$. In all other cases, $\Out(S)$ is cyclic.

The following is an immediate corollary of Theorem \ref{thm:main} and the main result of \cite{Man85a}.

\begin{corl}\label{cor:C}
Let $G$ be a finite group, and $p$ a prime such that $\cd(B_0(G))$ consists only of prime powers. Then there are at most $3$ primes dividing the degrees in $\cd(B_0(G))$.
\end{corl}

Problems on character degrees of finite groups have led to the study of the so called character degree graph $\Gamma(G)$ whose vertices are primes dividing the degree of some character of $G$, and two vertices $p, q$ are connected if there is $\chi\in\Irr(G)$ with $pq\mid\chi(1)$. Shortly after Manz's work, it was proved that $\Gamma(G)$ has at most three connected components, and if $G$ is solvable then it has at most two (see \cite[Theorems 4.2 and 6.4]{Lewis}). In \cite{MS05} an analogous graph $\Gamma(B)$ was introduced for a $p$-block $B$, where the degrees considered are only those of characters that lie in $\Irr(B)$. In \cite[Corollary C]{MS05} the authors show that if $G$ is $p$-solvable then $\Gamma(B)$ has at most three connected components and if $G$ is solvable then it has two, so it seems that $\Gamma(B)$ behaves somewhat similarly to $\Gamma(G)$. As pointed out by Moret\'o, it is interesting to speculate whether $\Gamma(B_0(G))$ has at most $3$ connected components in general, mimicking the situation in $\Gamma(G)$.

We prove Theorem \ref{thm:simple} in Section \ref{sec:simple} and we prove Theorem \ref{thm:main} and Corollary \ref{cor:C} in Section \ref{sec:main}.

\subsection*{Acknowledgements}
The results in this note were obtained while the author visited the Department of Mathematics of the  Rheinland-Pf\"alzische Technische Universit\"at (formerly TU Kaiserslautern). He thanks Gunter Malle for supervising his visit and for a thorough read of this manuscript, and the entire department for their warm hospitality. Furthermore, he would like to thank Alexander Moret\'o for his question and very useful conversations on the topic, and Annika Bartelt for clarifying some formulas for unipotent characters.

\section{Simple groups}\label{sec:simple}

The aim of this section is to prove Theorem \ref{thm:simple}. We start by recalling some classical results in number theory.

\begin{lem}[Zsigmondy's theorem]\label{lem:zsig}
Let $q$ be a prime and $n>1$ an integer. Then
\begin{enumerate}
\item there is a prime dividing $q^n-1$ that does not divide $q^m-1$ for all $m<n$ unless $q=2$ and $n=6$ or $n=2$ and $q+1$ is a power of $2$,
\item there is a prime dividing $q^n+1$ that does not divide $q^m+1$ for all $m<n$ unless $q=2$ and $n=3$.
\end{enumerate}
\end{lem}

\begin{lem}\label{lem:hb}
Suppose that $q$ is an odd prime and $q^n+1=2^s$ for some positive integers $n$ and $s$. Then $n=1$.
\end{lem}
\begin{proof}
See \cite[Chapter IX, Lemma 2.7]{HB82}.
\end{proof}

The following immediately follows from the previous lemmas.

\begin{lem}\label{lem:48}
Assume $q$ is a power of $2$ such that $q-1$ and $q+1$ are prime powers. Then $q\in\{4, 8\}$.
\end{lem}
\begin{proof}
Assume $q>8$ and that both $q-1$ and $q+1$ are prime powers. By Lemma \ref{lem:hb}, $q-1$ is a prime, and therefore $q+1$ is a power of $3$. By Lemma \ref{lem:zsig} there is a prime dividing $q+1$ that does not divide $2+1=3$, a contradiction.
\end{proof}

The next result is one of our main tools for discarding simple groups for Theorem \ref{thm:simple}.

\begin{lem}\label{lem:ppaldegrees}
If $G$ is not a $p$-solvable group then $|\cd(B_0(G))|\geq 3$, and if $p\geq 5$ there are at least $3$ character degrees in $\cd(B_0(G))$ not divisible by $p$. \end{lem}
\begin{proof}
This follows from the main results of \cite{M21} and \cite{GRSS20}.
\end{proof}

Next, we exclude most families of finite simple groups as candidates for Theorem \ref{thm:simple}.

\begin{pro}\label{pro:others}
Assume $S$ is not one of $\PSL_n(q), \PSU_n(q), \PSp_{2n}(q)$. Then there is some $\chi\in\Irr(B_0(S))$ whose degree is not a prime power.
\end{pro}
\begin{proof}
Assume first that $S$ is an alternating group $\fA_n$ for $n\geq 7$. If $n\leq 9$ this is easily checked in \cite{GAP}. If $n\geq 10$ then by the main result of \cite{BBOO01} there is at most one nonlinear representation of $\fA_n$ of prime power degree $q$. By Lemma \ref{lem:ppaldegrees} there is a character $\chi\in\Irr(B_0(G))$ of degree $\chi(1)\neq q$ so we are done in this case. 

If $S$ is a simple group appearing in cases (7)--(27) of \cite[Theorem 1.1]{MZ01} then this is also easily checked in \cite{GAP}.  Finally, if $S$ is a simple group not appearing in any of the cases (2)--(27) of \cite[Theorem 1.1]{MZ01} then this implies that $S$ is a simple group of Lie type and the nonlinear character of $S$ with prime power degree is the Steinberg character $\St_S$. Then by Lemma \ref{lem:ppaldegrees} there is a nonlinear character $\chi\in\Irr(B_0(S))$ with $\chi(1)\neq \St_S(1)$ so we are done.
\end{proof}

Thus we are left to deal with the groups $\PSL_n(q), \PSU_n(q)$ and $\PSp_{2n}(q)$. We begin with an easy observation.

\begin{pro}\label{pro:defining}
Let $S$ be a simple group of Lie type in characteristic $p$. Then Theorem \ref{thm:simple} holds for $(S,p)$.
\end{pro}
\begin{proof}
Assume $S$ is not one of the groups in Proposition \ref{pro:others}. The group $S=\PSp_4(2)'\cong\PSL_2(9)$ for $p=2$ can be checked in \cite{GAP}. Otherwise, by \cite[Theorem 3.3]{Cab18}, $\Irr(B_0(S))=\Irr(S)\setminus\{\St_S\}$ where $\St_S$ denotes the Steinberg character of $S$. Since $\St_S(1)$ is a prime power, $\cd(B_0(S))$ consists only of prime powers if and only if every character degree of $S$ is a prime power. By the main result of \cite{Man85b} we have that $S\in\{\SL_2(4), \SL_2(8)\}$.
\end{proof}

To deal with the remaining groups we will need to use so-called unipotent characters, introduced by G. Lusztig. In the case of $S=\PSL_n(q)$ or $\PSU_n(q)$, these are characters of $\SL_n(q)$ or respectively $\SU_n(q)$ but as argued in, for example, the first paragraph of \cite[Proposition 4.4]{GRSS20}, these characters contain $\zent{\SL_n(q)}$ or resp. $\zent{\SU_n(q)}$ in their kernels, and so they can be seen as characters of $S$, and furthermore they are contained in the principal block of $G$ if and only if they are in the principal block of $S$ (using \cite[Lemma 17.2]{CE04}). 

In this case, they are parametrized by partitions $n$ (see \cite[Section 4.3]{GM20}). Let $p$ be a prime dividing $|S|$ and let $e$ denote the order of $q$ modulo $p$ if $S=\PSL_n(q)$ and the order of $-q$ modulo $p$ if $S=\PSU_n(q)$. Further, let $r$ denote the remainder of $n$ divided by $m$. Following \cite{FS82} we have that a unipotent character of $S$ parametrized by the partition $\alpha$ of $n$ belongs to the principal $p$-block if its $e$-core is $(r)$.

If $S=\PSp_{2n}(q)$ then the unipotent characters are characters of $\Sp_{2n}(q)$ parametrized by certain symbols (see \cite[Section 4.4]{GM20}). Exactly as before, they can be seen as characters of $S$ and a unipotent character of $\Sp_{2n}(q)$ belongs to $B_0(\Sp_{2n}(q))$ if and only if it belongs to $B_0(S)$. 

\begin{lem}\label{lem:unipotent}
Let $G=\SL_n(q)$ or $\SU_n(q)$ and let $\chi$ be a unipotent character of $G$. Then $\chi(1)$ is not a prime power unless $\chi=1_G$ or $\chi=\St_G$.
\end{lem}
\begin{proof}
Let $\chi$ be the unipotent character parametrized by the partition $\alpha$ of $n$. It is easy to see in \cite[Propositions 4.3.1 and 4.3.5]{GM20} that $\chi(1)$ has nontrivial $q$-part and nontrivial $q'$-part unless $\alpha=(n)$ or $\alpha=(1^n)$, which correspond to $1_G$ and $\St_G$ respectively.
\end{proof}

\begin{pro}\label{pro:psl3}
Let $S=\PSL_n(q)$ or $\PSU_n(q)$ with $n\geq 3$ and let $p\nmid q$ be a prime dividing $|S|$. Then there is $\chi\in\Irr(B_0(S))$ with $\chi(1)$ not a prime power.
\end{pro}
\begin{proof}
.  Let $\eps=1$ if $S=\PSL_n(q)$ and $\eps=-1$ if $S=\PSU_n(q)$. By Lemma \ref{lem:ppaldegrees}, $S$ has to be one of the groups in cases (3) or (4) of \cite[Theorem 1.1]{MZ01}, so $n$ is an odd prime and the prime power degrees are $\St_S(1)=|S|_q$ and $(q^n-\eps)/(q-\eps)$. By Lemma \ref{lem:zsig}, $(q^n-\eps)/(q-\eps)$ can not be a power of $2$. 

If $q$ is odd then both prime power degrees are odd. If $p=2$ then the order of $\eps q$ modulo $p$ is necessarily $1$, and therefore all unipotent characters belong to $B_0(S)$, and we are done by Lemma \ref{lem:unipotent}. If $p$ is odd then \cite[Theorem B]{GMV19} guarantees the existence of a character of even degree in $B_0(S)$. Thus we assume $q$ is a power of $2$. 

Assume first $n\geq 5$. If $p\geq 5$ then \cite[Table 1]{GRSS20} produces a unipotent character in $B_0(S)$ different from $1_G$ and $\St_S$, which can not have prime power degree by Lemma \ref{lem:unipotent}. If $p=3$ then we argue identically with \cite[Table 3]{RSV20}.

We are left with the groups $\PSL_3(q)$ and $\PSU_3(q)$ with $q$ a power of $2$. If $p=3$ or $p\mid (q+\eps)$ then the last paragraph of the proof of \cite[Proposition 3.10]{RSV20} and the first paragraph of the proof of \cite[Proposition 4.5]{GRSS20} produces a character in $B_0(S)$ of degree $q^3-\eps$, which is not a prime power. If $p\geq 5$ then if $p\nmid (q+\eps)$ then by the order formula for $\PSL_3(q)$ and $\PSU_3(q)$ we have that either $p\mid (q-\eps)$ or $p\mid (q^2+\eps q+1)=(q^3-\eps)/(q-\eps)$.

In the first case, the order of $\eps q$ modulo $p$ is $1$, and it follows that the unipotent character defined by the partition $(1,2)$ belongs to $B_0(S)$ and has degree $q(q+\eps)$ by \cite[Propositions 4.3.1 and 4.3.5]{GM20}.

In the second case, we have that the prime power degrees of $S$ are $\St_S(1)=|S|_q$ and $(q^3-\eps)/(q-\eps)$ which must be a power of $p$. Since $p\geq 5$, Lemma \ref{lem:ppaldegrees} guarantees the existence of a character $\chi\in\Irr(B_0(S))$ of $p'$-degree different from $\St_S(1)$, so $\chi(1)$ is not a prime power.
\end{proof}

\begin{pro}\label{pro:psl2}
Let $S=\PSL_2(q)$ and let $p\nmid q$ be a prime dividing $|S|$ and assume first that every character degree in $B_0(S)$ is a prime power. Then $(S,p)$ is one of the pairs in Theorem \ref{thm:simple}.\end{pro}

\begin{proof}
If $q$ is odd then it is well known that $\cd(S)=\{1, q, q+1, q-1, (q\pm1)/2\}$ where the last sign depends on the congruence of $q$ modulo $4$. By Lemma \ref{lem:ppaldegrees} the set $\cd(B_0(S))$ has size at least 3, of which forces at least one of $q+1$ or $q-1$ to be a power of $2$. If $p=2$ then let $\chi\in\Irr(S)$ be a character of such degree. Notice that by the order formula for $\PSL_2(q)$, $\chi$ has $2$-defect zero and so it can not belong to the principal $2$-block of $S$.
This forces $p\neq 2$.  If $q+1$ is a power of $2$ the same argument works.

If $q$ is a power of $2$ then $\cd(S)=\{1, q, q+1, q-1\}$ which again forces $q+1$ or $q-1$ to be a prime power (and they both are prime powers only if $q\in\{4,8\}$ by Lemma \ref{lem:48}). Assume that $q>8$. If $r=q-1$ is a prime power then by Lemma \ref{lem:hb} in fact $r$ is a prime. A character of degree $r$ has $r$-defect zero and so it can not belong to the principal $r$-block, forcing $p\neq r$. If $r=q+1$ is a prime power then again by Lemma \ref{lem:zsig} we have that $r$ is a prime and we can mimick the previous argument to reach the desired conclusion.
\end{proof}

\begin{pro}\label{pro:psp}
Let $S=\PSp_{2n}(q)$ and let $p\nmid q$ be a prime dividing $|S|$. Then there is $\chi\in\Irr(B_0(S))$ with $\chi(1)$ not a prime power.
\end{pro}
\begin{proof}
 By Lemma \ref{lem:ppaldegrees}, we may assume that $S$ is one of the groups in cases (5) or (6) of \cite[Theorem 1.1]{MZ01}. Assume first that we are in case (5), so that the irreducible characters of $S$ whose degree is a prime power have degrees $\St_S(1)=|S|_q$ and $(q^n+1)/2$. Notice that $q^n+1$ can not be a power of $2$ by Lemma \ref{lem:zsig}.

If $p$ is odd then there must exist a character $\chi\in\Irr(B_0(S))$ of even degree by \cite[Theorem B]{GMV19}. If $p=2$ then by \cite[Theorem 21.14]{CE04}, all unipotent characters belong to $B_0(S)$. Consider the unipotent character $\chi$ parametrized by the symbol $\binom{0\, n}{1}$. By \cite[Proposition 4.4.15]{GM20} and using the order formula for $\Sp_{2n}(q)$, it is easy to see that $\chi(1)_q>1$ and $\chi(1)_{q'}>1$ and so $\chi(1)$ is not a prime power.

If we are in case (6) of \cite[Theorem 1.1]{MZ01}, the characters of prime power degree have degrees $3^l$ and $(3^n-1)/2$ which is only a power of $2$ if $n=2$ so $S=\PSp_4(3)$, which is checked in \cite{GAP}. If $n>2$ then the same argument as before applies.
\end{proof}

\textit{Proof of Theorem \ref{thm:simple}.}
The only if direction is done in Propositions \ref{pro:others}, \ref{pro:defining}, \ref{pro:psl3}, \ref{pro:psl2}, \ref{pro:psp}. 

For the if direction the cases (iii)--(v) can be checked in \cite{GAP}. Assume first that $S=\PSL_2(q)$ where $q+1$ is a power of $2$ and let $p\neq \{2, q\}$. Then the characters of degree $(q-1)$ and $(q-1)/2$ (if they exist) have $p$-defect zero by the order formula for $\PSL_2(q)$, so they can not belong to the principal $p$-block. This forces $$\cd(B_0(S))\sbs\{1, q, q+1, (q+1)/2\}.$$ The case where $q$ or $q-1$ is a power of $2$ is done analogously.
\qed

\section{Theorem \ref{thm:main} and Corollary \ref{cor:C}}\label{sec:main}

Our notation for this section follows \cite{Nav98} and \cite{Nav18}.

\begin{lem}\label{lem:normalsubgroups}
Assume $\cd(B_0(G))$ consists only of prime powers. If $N\normal G$ then $B_0(N)$ and $B_0(G/N)$ consist only of prime powers.
\end{lem}
\begin{proof}
For all $\theta\in\Irr(B_0(N))$ there is some $\chi\in\Irr(B_0(G))$ such that $\chi_N$ contains $\theta$ by \cite[Theorem 9.4]{Nav98}, and $\theta(1)$ divides $\chi(1)$ by standard Clifford theory, and the first claim follows. For the second claim recall that $\Irr(B_0(G/N))\sbs\Irr(B_0(G))$.
\end{proof}

\begin{thm}\label{thm:nonpsolvable}
Assume that $G$ is not $p$-solvable and that $\cd(B_0(G))$ consists only of prime powers. Then there is some $\oh{p'}G\sbs M\normal G$ with $M/\oh{p'}G=H\times S$ where $H$ is an abelian $p$-group and $(S,p)$ is one of the pairs in Theorem \ref{thm:simple}. Also, $G/M$ is isomorphic to a subgroup of $\Out(S)$.
\end{thm}

\begin{proof}
Arguing by induction on $|G|$, we may assume $\oh{p'}G=1$. Let $E$ be the layer of $G$. We claim that $E$ is quasisimple. Indeed, write $E=K_1\cdots K_t$ for components $K_1,\dots, K_t$, assume $t>1$ and let $Z=\zent{E}$. Then $Z\normal G$ and by \cite[6.5.6]{KS} $E/Z=E_1\times\dots\times E_t$ where $E_i=K_iZ/Z$. By Lemma \ref{lem:normalsubgroups}, $B_0(E/Z)$ consists only of prime powers, and so does $B_0(E_i)$. Now since the $E_i$'s are not $p$-solvable, by Lemma \ref{lem:ppaldegrees} and \cite[Lemma 5.2]{M21}, for each $E_i$ there are characters $\theta,\eta\in\Irr(B_0(E_i))$ of coprime degree. Let $\theta\in\Irr(B_0(E_1))$ and $\eta\in\Irr(B_0(E_2))$ be nonlinear and such that $\theta(1)$ and $\eta(1)$ are coprime and consider $\psi=\theta\times\eta\times 1_{E_3}\times\dots\times 1_{E_t}\in\Irr(B_0(E/Z))$. Then $\psi(1)$ is not a prime power, a contradiction. This forces $E$ to be quasisimple and again by Lemma \ref{lem:normalsubgroups}, $(E/\zent E, p)$ is one of the pairs from Theorem \ref{thm:simple}. Furthermore, the $p$-complement of $\zent E$ is a normal $p'$-subgroup of $G$, and since $\oh{p'}G=1$ we have that $\zent E$ is a $p$-group.

 Now let $C=\cent G E$ so that $G/C$ is almost simple with socle $S=E/\zent E$ (arguing as in Step 7 of \cite[Theorem 2.10]{GMS}).  Now $C\normal G$ and $M=CE\normal G$ is a central product because $[C, E]=1$. Arguing as before, it is easy to see that $B_0(C)$ can not contain nonlinear characters, so $C/\oh{p'}C$ is an abelian $p$-group by \cite[Theorem 6.10]{Nav98}. Now since $\oh{p'}C\normal G$ we get that $C$ is an abelian $p$-group and $\zent E=C\cap E$. Now if the Schur multiplier of $S$ has order not divisible by $p$ then $\zent E=1$ and $M=C\times S$. By \cite[Tables 24.2 and 24.3]{MT}, the only pair $(S, p)$ for which the Schur multiplier of $S$ has order divisible by $p$ is $(\fA_5, 2)$, and it is easily checked in \cite{GAP} that the universal central extension $2.\fA_5$ has a character of degree $6$ in its principal  $2$-block, so in all cases $M=C\times S$.
Finally, since $G/C$ is almost simple with socle $S$ we have that $G/M$ is isomorphic to a subgroup of $\Out(S)$.
\end{proof}

\textit{Proof of Theorem \ref{thm:main}.} If $G$ is $p$-solvable then by \cite[Theorem 10.20]{Nav98} we have that $\Irr(B_0(G))=\Irr(G/\oh{p'}G)$ and therefore $G/\oh{p'}G$ is one of the groups from \cite{Man85a}. Otherwise, apply Theorem \ref{thm:nonpsolvable}.\qed

\medskip

If $2^n+1$ is a prime, then it is well known that $n$ is a power of $2$. It is also well known that if $2^n-1$ is a prime then $n$ is a prime. Since the automorphism group of $S=\SL_2(2^n)$ is a cyclic group of order $n$, if we are in case (ii) of Theorem \ref{thm:simple} then $\Out(S)$ has prime power order.

\medskip

\textit{Proof of Corollary \ref{cor:C}}.
We may assume $G$ is as in case (ii) of Theorem \ref{thm:main}. If $\chi\in\Irr(B_0(G))$ has degree $\chi(1)=r^t$ for some prime $r$ and $\chi_S\neq 1_S$ then $\chi_S$ contains characters of degree a power of $r$ in $B_0(S)$ by \cite[Theorem 9.4]{Nav98}. Otherwise $\chi_S=1_S$ and $\chi_M$ contains only linear characters so $\chi(1)$ divides $|G/M|$ by \cite[Theorem 5.12]{Nav18}. Now, $G/M$ is isomorphic to a subgroup of $\Out(S)$ which has prime power order in all cases from Theorem \ref{thm:simple}, so we are done.\qed

\end{document}